\documentclass[11pt,a4paper]{article}
\usepackage[utf8]{inputenc}

\usepackage{amssymb,amscd,amsfonts,amsbsy,amsmath,amsthm,amsrefs}
\usepackage{dsfont}
\usepackage{enumerate}
\usepackage{mathrsfs}
\usepackage{epsf,epsfig,esint}
\usepackage{pdfsync}
\usepackage[all]{xy}
\usepackage{pdfpages}

\newtheorem{proposition}{Proposition}[section]
\newtheorem{theorem}[proposition]{Theorem}
\newtheorem{lemma}[proposition]{Lemma}

\theoremstyle{definition}
\newtheorem{definition}[proposition]{Definition}

\theoremstyle{remark}
\newtheorem{remark}[proposition]{Remark}

\numberwithin{equation}{section}

\begin{document}

\title{The Hall problem in domains}
\author{Sylvie Monniaux\footnote{Aix-Marseille Univ., CNRS, I2M UMR7373, Marseille, France 
- {\tt sylvie.monniaux@univ-amu.fr}} \footnote{France-Australia Mathematical Sciences and Interactions ANU - 
CNRS International Research Laboratory - {\tt sylvie.monniaux@cnrs.fr}}}

\date{In memory of Hermann Sohr}


\maketitle

\abstract{In this paper, we develop a framework based on differential forms
that enables us to deal with the Hall problem on domains in any dimension $n\ge 2$.
In the case of smooth bounded domains of ${\mathds{R}}^n$, we prove local existence of mild solutions in subcritical spaces.}

\section{Introduction}

The Hall problem for a magnetic field $b$ in a domain $\Omega\subset{\mathds{R}}^3$ on a 
time interval $(0,T)$ ($0<T\le \infty$) as considered in \cite{ADFL11} (with all constants equal to 1) 
reads
\begin{equation}
\label{hp}\tag{HP}
\partial_t b-\Delta b=-{\rm curl}\,({\rm curl}\,b\times b)\quad\mbox{ in }(0,T)\times\Omega
\end{equation}
where the {\sl magnetic field} (in the absence of magnetic monopole) is denoted by 
$b:(0,T)\times\Omega\to{\mathds{R}}^3$. The equation of \eqref{hp} 
describes the evolution of the magnetic field following the so-called {\sl induction} equation with
the {\sl Hall effect} ${\rm curl}\,({\rm curl}\,b\times b)$.

We assume in all what follows that $\Omega$ is a bounded Lipschitz domain.
The physically relevant boundary conditions for \eqref{hp} as stated in \cite{ADFL11}
are the perfectly conducting wall boundary conditions for the magnetic field; namely
\begin{equation}
\label{bdry-cond}\tag{BC}
\left\{
\begin{array}{rclcl}
\nu\cdot b&=&0&\mbox{ on }&(0,T)\times\partial\Omega\\
\nu\times{\rm curl}\,b+\nu\times({\rm curl}\,b\times b)&=&0&\mbox{ on }&(0,T)\times\partial\Omega,
\end{array}
\right.
\end{equation}
where $\nu(x)$ denotes the exterior unit normal vector at a point $x\in\partial\Omega$.
Denoting by $f$ the nonlinearity $-{\rm curl}\,b\times b$, the system \eqref{hp}-\eqref{bdry-cond} reduces to
\begin{equation}
\label{eq:b,f}\tag{LHP}
\left\{
\begin{array}{rclcl}
\partial_t b-\Delta b&=&{\rm curl}\,f &\mbox{ in }&(0,T)\times\Omega\\
\nu\cdot b&=&0&\mbox{ on }&(0,T)\times\partial\Omega\\
\nu\times{\rm curl}\,b&=&\nu\times f&\mbox{ on }&(0,T)\times\partial\Omega\\
\end{array}
\right.
\end{equation}
Note at this point that for $1<p<\infty$ and $f\in L^p(\Omega,{\mathds{R}}^3)$ with ${\rm curl}\,f\in L^p(\Omega,{\mathds{R}}^3)$, 
the quantity $\nu\times f$ at the boundary exists as a distribution in $B_{p,p}^{-\frac{1}{p}}(\partial\Omega,{\mathds{R}}^3)$ in the
following sense:

Let $\varphi\in B_{p',p'}^{\frac{1}{p}}(\partial\Omega,{\mathds{R}}^3)$, where $p'\in (1,\infty)$ is the conjugate of $p$, $i.e.$, 
$\frac{1}{p}+\frac{1}{p'}=1$: denote by $\Phi$ a $W^{1,p'}$ extension of $\varphi$ in $\Omega$, $i.e.$, 
$\Phi\in W^{1,p'}(\Omega,{\mathds{R}}^3)$ and ${\rm Tr}_{|_{\partial\Omega}}(\Phi)=\varphi$. We define
\begin{equation}
\label{eq:nxf}
\langle \nu\times f,\varphi\rangle_{\partial\Omega}:=\langle {\rm curl}\,f, \Phi\rangle_\Omega -\langle f,{\rm curl}\,\Phi\rangle_\Omega
\end{equation}
where $\langle \cdot,\cdot\rangle_{\partial\Omega}$ denotes the duality pairing 
$\bigl(B_{p,p}^{-\frac{1}{p}}(\partial\Omega,{\mathds{R}}^3),B_{p',p'}^{\frac{1}{p}}(\partial\Omega,{\mathds{R}}^3)\bigr)$
and $\langle \cdot,\cdot\rangle_{\Omega}$ denotes the duality pairing $\bigl(L^p(\Omega,{\mathds{R}}^3),L^{p'}(\Omega,{\mathds{R}}^3)\bigr)$.
We remark that this definition agrees with the usual integration by parts if 
$f\in {\mathscr{C}}(\overline{\Omega},{\mathds{R}}^3)\cap {\mathscr{C}}^1(\Omega,{\mathds{R}}^3)$.
It is easy to check that \eqref{eq:nxf} is independent of the choice of the extension $\Phi$ of $\varphi$.

The goal of this note is to write (and solve) the problem \eqref{eq:b,f} in the language of differential forms of any degree, extending it 
to any dimension $n\ge 2$ (see Section~\ref{sec:LHP}). The difficulty of giving the structure of a mild solution lies in the fact that the 
boundary condition on $\partial\Omega$ is not homogeneous. We will also prove the existence of mild solutions of the nonlinear Hall 
problem in domains in subcritical spaces (see Section~\ref{sec:HPn}).

\section{Differential forms}

A differential form on $\Omega\subset{\mathds{R}}^n$ is a function from $\Omega$ with values in the exterior algebra
$\Lambda=\Lambda^0\oplus \Lambda ^1\oplus ... \oplus \Lambda^n$ of ${\mathbb{R}}^n$. The 
space of $\ell$-vectors $\Lambda^\ell$ ($1\le \ell\le n$) is the span of $\bigl\{e_J, J\subset\{1,2,...,n\}, |J|=\ell\bigr\}$ where
\[
e_J=e_{j_1}\wedge e_{j_2}\wedge ...\wedge e_{j_\ell}\mbox{ for } J=\{e_{j_1},e_{j_2},..., e_{j_\ell}\}
\mbox{ with }1\le j_1<j_2<...<j_\ell\le n.
\]
The $0$-vectors consist in scalars. That way, a differential form $u:\Omega\to \Lambda$ can be represented by
\[
u=\sum_{J\subset\{1,2,...,n\}}u_Je_J=u_0+
\sum_{\underset{1\le j_1<j_2<...<j_\ell\le n}{\ell=1}}^nu_{j_1,...,j_\ell}\,e_{j_1}\wedge e_{j_2}\wedge...\wedge e_{j_\ell},
\]
where $u_0$ and $u_{j_1,...,j_\ell}$, for any $1\le \ell\le n$ and any $1\le j_1<j_2<...<j_\ell\le n$, maps $\Omega$ to 
${\mathds{R}}$.
Here, $\wedge$ is the exterior product in the exterior algebra $\Lambda$, 
$\lrcorner\,$ is the interior product (or contraction). For two differential forms $u=\sum_J u_J e_J$ and $v=\sum_Kv_Ke_K$, 
the notation $(u,v)$ stands for
$\sum_J u_Jv_J$ and for a 1-form $a$, we have the relation $(a\wedge u,v)=(u, a\lrcorner\,v)$. 

We denote by $d=\nabla\wedge =\sum_{i=1}^n\partial_i e_i \wedge$ 
the exterior derivative (it satisfies $d^2=0$) and $\delta=-\nabla \lrcorner\,=-\sum_{i=1}^n\partial_i e_i\lrcorner\,$
represents the interior derivative (or co-derivative) acting on differential forms from $\Omega$ to 
the exterior algebra $\Lambda$ ($\delta^2=0$ as well). 

In dimension $3$, we have these correspondences:
\begin{align*}
	d :\ & \Lambda^0\sim{\mathds{R}} \xrightarrow[]{\nabla} \Lambda^1\sim{\mathds{R}}^3 \xrightarrow[]{{\rm curl}} 
	\Lambda^2\sim{\mathds{R}}^3 \xrightarrow[]{{\rm div}} \Lambda^3\sim {\mathds{R}}\\
	& \Lambda^0\sim{\mathds{R}} \xleftarrow[-{\rm div}]{} \Lambda^1\sim{\mathds{R}}^3 \xleftarrow[{\rm curl}]{} 
	\Lambda^2\sim{\mathds{R}}^3 \xleftarrow[-\nabla]{} \Lambda^3\sim {\mathds{R}}\ : \ \delta
\end{align*}
This means that, when acting on ${\mathscr{C}}^\infty(\overline{\Omega},\Lambda)$, $d$ acts as a gradient on 0-forms, 
as a ${\rm curl}\,$ on 1-forms, as a divergence on 2-forms and maps 3-forms to 0. 

For an operator $A$ acting on functions with values in $\Lambda$, ${\sf N}_p(A)$, ${\sf R}_p(A)$ and ${\sf D}_p(A)$ denote 
respectively the kernel of the operator A in $L^p(\Omega,\Lambda)$, its range in $L^p$ and its domain in $L^p$. We have
(for $A=d$ and $A=\delta$):
\[
{\sf D}_p(d)=\{u\in L^p(\Omega,\Lambda); d u\in L^p(\Omega,\Lambda)\} \quad \mbox{ and }\quad
{\sf D}_p(\delta)=\{u\in L^p(\Omega,\Lambda); \delta u\in L^p(\Omega,\Lambda)\}.
\]
The following integration by parts holds for $u,v\in {\mathscr{C}}^\infty(\overline{\Omega},\Lambda)$:
\begin{equation}
\label{eq:int/parts}
\int_{\Omega}(du,v)\,{\rm d}x=\int_{\Omega}(u,\delta v)\,{\rm d}x +\int_{\partial\Omega}(\nu\wedge u,v)\,{\rm d}\sigma
=\int_{\Omega}(u,\delta v)\,{\rm d}x +\int_{\partial\Omega}(u,\nu\lrcorner\,v)\,{\rm d}\sigma.
\end{equation}
Let us emphasize that, as in the 3D case, for $u\in {\sf D}_p(d)$, the quantity $\nu\wedge u$ on $\partial\Omega$ exists
as a distribution in $B^{-\frac{1}{p}}_{p,p}(\partial\Omega,\Lambda)$. And for $v\in {\sf D}_p(\delta)$, $\nu\lrcorner\,v$
exists as a distribution in $B^{-\frac{1}{p}}_{p,p}(\partial\Omega,\Lambda)$.
Indeed, let $\varphi \in B^{\frac{1}{p}}_{p',p'}(\partial\Omega,\Lambda)$:
we denote by $\Phi$ a $W^{1,p'}$ extension of $\varphi$ in $\Omega$, $i.e.$, $\Phi \in W^{1,p'}(\Omega,\Lambda)$ and
${\rm Tr}_{|_{\partial\Omega}}\Phi=\varphi$ (recall that $\frac{1}{p}=1-\frac{1}{p'}$). Then we define
\[
\langle \nu\wedge u,\varphi\rangle_{\partial\Omega}:=\langle du,\Phi\rangle_\Omega -\langle u,\delta \Phi\rangle_\Omega
\]
and
\[
\langle\nu\lrcorner\,v, \varphi\rangle_{\partial\Omega}:=-\langle \delta v, \Phi\rangle_\Omega + \langle v, d\Phi\rangle_\Omega,
\]
where $\langle\cdot,\cdot\rangle_{\partial\Omega}$ stands for the 
$B^{-\frac{1}{p}}_{p,p}(\partial\Omega,\Lambda)-B^{\frac{1}{p}}_{p',p'}(\partial\Omega,\Lambda)$ pairing and
$\langle \cdot,\cdot \rangle_\Omega$ is used for the $L^p(\Omega,\Lambda)-L^{p'}(\Omega,\Lambda)$ pairing.
It is not difficult to see that these definitions are independent of the choice of the extension $\Phi$ of $\varphi$.

Let us now define the operators ``with boundary conditions" $\underline{d}$ and $\underline{\delta}$:
\[
{\sf D}_p(\underline{d})=\{u\in L^p(\Omega,\Lambda); d u\in L^p(\Omega,\Lambda), \nu\wedge u=0\mbox{ on }\partial\Omega\},
\quad \underline{d}u=du, \mbox{ for } u\in  {\sf D}_p(\underline{d})
\]
and 
\[
{\sf D}_p(\underline{\delta})=\{u\in L^p(\Omega,\Lambda); \delta u\in L^p(\Omega,\Lambda), \nu\lrcorner\,u=0\mbox{ on }\partial\Omega\},
\quad \underline{\delta}u=\delta u, \mbox{ for } u\in  {\sf D}_p(\underline{\delta}).
\]
Thanks to the integration by parts \eqref{eq:int/parts}, it is immediate that $(\underline{\delta}, {\sf D}_{p'}(\underline{\delta}))$ is
the adjoint of $(d, {\sf D}_p(d))$ and $(\underline{d}, {\sf D}_{p'}(\underline{d}))$ is the adjoint of $(\delta, {\sf D}_p(\delta))$.

Assume that $u\in {\sf D}_p(\underline{d})$: since $d^2=0$, $du$ belongs to ${\sf D}_p(d)$
and we have that $\nu\wedge du=0$ in $B^{-\frac{1}{p}}_{p,p}(\partial\Omega,\Lambda)$. 
Indeed, let $\varphi\in B^{\frac{1}{p}}_{p',p'}(\partial\Omega,\Lambda)$ and denote by $\Phi$ a $W^{1,p'}$ extension 
of $\varphi$ in $\Omega$. Applying twice the integration by parts formula \eqref{eq:int/parts} we have that
\begin{equation}
\label{eq:nuxdu}
\langle \nu\wedge du,\varphi\rangle_{\partial\Omega} =\langle d(du), \Phi\rangle_\Omega -\langle du,\delta\Phi\rangle_\Omega
=-\langle u,\delta(\delta\Phi)\rangle_\Omega -\langle\nu\wedge u, {\rm Tr}_{|_{\partial\Omega}}(\delta\Phi)\rangle_{\partial\Omega}=0.
\end{equation}
The same applies to $u\in {\sf D}_p(\underline{\delta})$: $\nu\lrcorner\,\delta u=0$ in $B^{-\frac{1}{p}}_{p,p}(\partial\Omega,\Lambda)$, 
with the same kind of calculus:
\begin{equation}
\label{eq:nu.deltau}
\langle \nu\lrcorner\,\delta u, \varphi\rangle_{\partial\Omega}= -\langle \delta(\delta u),\Phi\rangle_\Omega+\langle \delta u,d\Phi\rangle_\Omega
=\langle u, d(d\Phi)\rangle_\Omega -\langle\nu\lrcorner\,u,\delta \Phi\rangle_{\partial\Omega} =0.
\end{equation}
For more on differential forms, we refer to \cite{McIM18} \S2.3 or \cite{AMcI04} and the references therein.
See also \cite{MM00} \S2 or \cite{MMT01} \S4.

It has been established in \cite{M25} that the following Hodge decompositions hold 
in $L^p(\Omega,\Lambda)$ when $\Omega\subset{\mathds{R}}^n$ is a bounded ${\mathscr{C}}^1$ domain, 
with bounded accompanying projections. That is
\begin{equation}
\label{eq:hodge-dec}
L^p(\Omega,\Lambda)={\sf N}_p(d)\oplus{\sf R}_p(\underline{\delta})
= {\sf N}_p(\delta)\oplus{\sf R}_p(\underline{d})
\end{equation}
hold for $\Omega\subset{\mathds{R}}^n$ whenever $p\in(1,\infty)$. If the domain has only a Lipschitz boundary 
$\partial\Omega$, \eqref{eq:hodge-dec} holds only for $p$ in an open interval $(p_H,p^H)$
containing $[\frac{2n}{n+1},\frac{2n}{n-1}]$ (see \cite{McIM18} \S7).

We denote by ${\mathbb{Q}}$ the 
bounded projection from $L^p(\Omega,\Lambda)$ to ${\sf N}_p(\delta)$: for all $f\in{\sf D}_p(d)$
we have that $d({\mathbb{Q}}f)=df$ in $\Omega$ and $\nu\wedge ({\mathbb{Q}}f)=\nu\wedge f$ on $\partial\Omega$.

In the same spirit, we denote by ${\mathbb{P}}$ the 
bounded projection from
$L^p(\Omega,\Lambda)$ to ${\sf N}_p(d)$: for all $g\in{\sf D}_p(\delta)$, we have that $\delta({\mathbb{P}}g)=\delta g$ 
in $\Omega$ and $\nu\lrcorner\, ({\mathbb{P}}g)=\nu\lrcorner\, g$ on $\partial\Omega$.

 Applying \cite{McIM18}, Theorem~5.1, thanks to the Hodge decompositions
\eqref{eq:hodge-dec}, the Hodge-Dirac operators in $L^p(\Omega,\Lambda)$, $D_{\bot}=\underline{d}+\delta$ with domain
${\sf D}_p(\underline{d})\cap {\sf D}_p(\delta)$ and $D_{\|}=d+\underline{\delta}$ with domain 
${\sf D}_p(d)\cap {\sf D}_p(\underline{\delta})$, admit a bounded $S_\mu^o$ 
holomorphic functional calculus in $L^p(\Omega,\Lambda)$ for all $1<p<\infty$ and all $\mu\in(0,\frac{\pi}{2})$. 
As before, if $\Omega$ is only a bounded Lipschitz domain, this holds for $p\in(p_H,p^H)$. 

Moving on, following \cite{McIM18} Corollary~8.1, the Hodge Laplacian $M:=D_\bot^2$ admits a bounded 
$S_{\mu+}^o$ holomorphic functional calculus in $L^p(\Omega,\Lambda)$ and in ${\sf R}_p(\underline{d})$
for all $1<p<\infty$ and all $\mu\in(0,\frac{\pi}{2})$. With the same arguments, $\tilde{M}:= D_\|^2$
admits a bounded $S_{\mu+}^o$ holomorphic functional calculus in $L^p(\Omega,\Lambda)$ and in 
${\sf R}_p(\underline{\delta})$ for all $1<p<\infty$ and all $\mu\in(0,\frac{\pi}{2})$ ($p_H<p<p^H$ if 
$\Omega$ is only Lipschitz).

\section{The linearised Hall problem}
\label{sec:LHP}

We are now in position to generalise \eqref{eq:b,f} to differential forms of any degree in a bounded ${\mathscr{C}}^1$ domain 
$\Omega\subset{\mathds{R}}^n$, $n\ge 2$. The following two systems are dual to each other:
\begin{equation}
\label{eq:Glhp1}
\left\{
\begin{array}{rclcl}
\partial_tw-\Delta w&=&df&\mbox{ in }&(0,\infty)\times\Omega\\
\nu\wedge w&=&0&\mbox{ on }&(0,\infty)\times\partial\Omega\\
\nu\wedge\delta w&=&\nu\wedge f&\mbox{ on }&(0,\infty)\times\partial\Omega
\end{array}
\right.
\end{equation}
with $f\in {\rm D}_p(d)$ and
\begin{equation}
\label{eq:Glhp2}
\left\{
\begin{array}{rclcl}
\partial_tv-\Delta v&=&\delta g&\mbox{ in }&(0,\infty)\times\Omega\\
\nu\lrcorner\,v&=&0&\mbox{ on }&(0,\infty)\times\partial\Omega\\
\nu\lrcorner\,dv&=&\nu\lrcorner\, g&\mbox{ on }&(0,\infty)\times\partial\Omega
\end{array}
\right.
\end{equation}
with $g\in {\rm D}_p(\delta)$. When restricted to $n=3$ and $f$ a 1-form, looking for 2-forms $w$ solution of 
\eqref{eq:Glhp1} amounts to solving \eqref{eq:b,f}. Replacing $f$ with ${\mathbb{Q}}f$ and $g$ with ${\mathbb{P}}g$,
we may assume that $\delta f=0$ and $dg=0$.

\begin{lemma}
\label{lem:dw=0}
Assume that $w\in {\mathscr{C}}^1((0,\infty),{\mathscr{C}}^2(\Omega,\Lambda))
\cap{\mathscr{C}}([0,\infty),{\mathscr{C}}^1(\overline{\Omega},\Lambda))$ satisfies 
\eqref{eq:Glhp1}. If $dw(0)=0$, then $dw(t)=0$ for all $t\ge0$.

This result is also valid for \eqref{eq:Glhp2}: for $v\in {\mathscr{C}}^1((0,\infty),{\mathscr{C}}^2(\Omega,\Lambda))
\cap{\mathscr{C}}([0,\infty),{\mathscr{C}}^1(\overline{\Omega},\Lambda))$ satisfying \eqref{eq:Glhp2},
if $\delta v(0)=0$, then $\delta v(t)=0$ for all $t\ge0$.
\end{lemma}

\begin{proof}
Let $a:=dw$ for $w$ satisfying \eqref{eq:Glhp1}: $a\in {\mathscr{C}}^1((0,\infty),{\mathscr{C}}^1(\Omega,\Lambda))$ 
and satisfies $\partial_ta-\Delta a=0$ in $(0,\infty)\times\Omega$ (since $d(df)=0$). 
Thanks to \eqref{eq:nuxdu}, we have that $\nu\wedge a=0$ on $(0,\infty)\times\partial\Omega$.
We now prove that $\nu\wedge \delta a=0$ on $(0,\infty)\times\partial\Omega$: we have that
$\nu\wedge (f-\delta w)=0$ so that by \eqref{eq:nuxdu}, this gives $\nu\wedge d(f-\delta w)=0$. Moreover, since $\nu\wedge w=0$,
we also have that $\nu\wedge \partial_tw=0$. Using the fact that $\delta a=\delta (dw)=-\Delta w -d\delta w =-\partial_tw+d(f-\delta w)$,
we obtain $\nu\wedge\delta a=0$ as claimed. This gives then that $a(t)=e^{-tM}a(0)$ for all $t\ge 0$ where 
$M=(\underline{d}+\delta)^2$: if $a(0)=dw(0)=0$, then $dw(t)=a(t)=0$ for all $t\ge 0$.

\noindent
The same reasoning applies to $b=\delta v$ for $v$ satisfying \eqref{eq:Glhp2}: 
$b\in {\mathscr{C}}^1((0,\infty),{\mathscr{C}}^1(\Omega,\Lambda))$ and satisfies $\partial_tb-\Delta b=0$ in 
$(0,\infty)\times\Omega$ (since $\delta(\delta g)=0$). Thanks to \eqref{eq:nu.deltau}, we have that $\nu\lrcorner\, b=0$ on 
$(0,\infty)\times\partial\Omega$.
We now prove that $\nu\lrcorner\, db=0$ on $(0,\infty)\times\partial\Omega$: we have that
$\nu\lrcorner\, (g-dv)=0$ so that by \eqref{eq:nu.deltau}, this gives $\nu\wedge \delta(g-dv)=0$. Moreover, since $\nu\lrcorner\, v=0$,
we also have that $\nu\lrcorner\, \partial_tv=0$. Using the fact that $db=d(\delta v)=-\Delta v -\delta (dv) =-\partial_tv+\delta(g-dv)$,
we obtain $\nu\lrcorner\,db=0$ as claimed. This gives then that $b(t)=e^{-t\tilde{M}}b(0)$ for all $t\ge 0$ where 
$\tilde{M}=(\underline{\delta}+d)^2$: if $b(0)=\delta v(0)=0$, then $\delta v(t)=b(t)=0$ for all $t\ge 0$.
\end{proof}

We will now construct solutions for \eqref{eq:Glhp1} and \eqref{eq:Glhp2}.

We are now in position to state our result on existence of solutions of \eqref{eq:Glhp1} and \eqref{eq:Glhp2}.

\begin{theorem}
Let $1<p<\infty$ if $\Omega$ is ${\mathscr{C}}^1$ and $p_H<p<p^H$ is $\Omega$ is Lipschitz.
\begin{itemize}
\item[(i)]
Let $X=L^p(\Omega,\Lambda)$ or $X={\sf R}_p(\underline{d})$.
For all $f\in {\mathscr{C}}(0,\infty;L^p(\Omega,\Lambda))$ such that $s\mapsto \sqrt{s}f(s)\in L^\infty(0,\infty;L^p(\Omega,\Lambda))$ 
and all $w_0\in X$, the linearised Hall problem \eqref{eq:Glhp1}
admits a unique solution $w\in {\mathscr{C}}([0,\infty);X)$ given by
\begin{equation}
\label{eq:w}
w(t)=e^{-tM}w_0+\underline{d}\int_0^t e^{-(t-s)M}{\mathbb{Q}}f(s)\,{\rm d}s, \quad t\ge 0.
\end{equation}
\item[(ii)]
Let $X=L^p(\Omega,\Lambda)$ or $X={\sf R}_p(\underline{\delta})$.
For all $g\in {\mathscr{C}}(0,\infty;L^p(\Omega,\Lambda))$ such that $s\mapsto \sqrt{s}g(s)\in L^\infty(0,\infty;L^p(\Omega,\Lambda))$ 
and all $v_0\in X$, the unique solution 
$v\in {\mathscr{C}}([0,\infty);X)$ of the linearised Hall problem \eqref{eq:Glhp2} is given by
\begin{equation}
\label{eq:v}
v(t)=e^{-t\tilde{M}}v_0+\underline{\delta}\int_0^t e^{-(t-s)\tilde{M}}{\mathbb{P}}g(s)\,{\rm d}s, \quad t\ge 0.
\end{equation}
\end{itemize}
\end{theorem}

\begin{proof}
$(i)$ For $t\ge 0$, we write $w_1(t)=e^{-tM}w_0$ and $w_2(t) =\underline{d}\,W(t)$ where 
\[
W(t)=\int_0^te^{-(t-s)M}{\mathbb{Q}}f(s)\,{\rm d}s.
\]
The first part $(i)$ of the theorem states that $w=w_1+w_2$ is a solution of \eqref{eq:Glhp1}. By classical semigroup theory, it is clear
that $w_1\in{\mathscr{C}}((0,\infty);{\sf D}_p(M))\cap {\mathscr{C}}^1((0,\infty);L^p(\Omega,\Lambda))\cap 
{\mathscr{C}}([0,\infty);L^p(\Omega,\Lambda))$ is solution of 
\[
\left\{
\begin{array}{rclcl}
\partial_tw_1-\Delta w_1&=&0&\mbox{ in }&(0,\infty)\times\Omega\\
\nu\wedge w_1&=&0&\mbox{ on }&(0,\infty)\times\partial\Omega\\
\nu\wedge\delta w_1&=&0&\mbox{ on }&(0,\infty)\times\partial\Omega\\
w_1(0)&=&w_0&\mbox{ in }&\Omega.
\end{array}
\right.
\] 
If $w_0\in {\sf R}_p(\underline{d})$, then $w_0=\underline{d}\alpha_0$ for some $\alpha_0\in {\sf D}_p(\underline{d})$ and
$w_1(t)=e^{-tM}w_0=e^{-tM}(\underline{d}\alpha_0)=\underline{d}e^{-tM}\alpha_0$ with $e^{-tM}\alpha_0\in {\sf D}_p(\underline{d})$
so that $w_1\in {\mathscr{C}}([0,\infty);{\sf R}_p(\underline{d}))$.

Duhamel's formula states that $W$ is solution of 
\[
\left\{
\begin{array}{rclcl}
\partial_t W-\Delta W&=&{\mathbb{Q}}f&\mbox{ in }&(0,\infty)\times\Omega\\
\nu\wedge W&=&0&\mbox{ on }&(0,\infty)\times\partial\Omega\\
\nu\wedge\delta W&=&0&\mbox{ on }&(0,\infty)\times\partial\Omega\\
W(0)&=&0&\mbox{ in }&\Omega,
\end{array}
\right.
\] 
and for all $T>0$, belongs to $L^q(0,T;{\sf D}_p(M))$ for $1<q<2$. Indeed, since $M$ admits a bounded holomorphic functional
calculus in $L^p(\Omega,\Lambda)$, it has also the maximal $L^q$ regularity property in $L^p(\Omega,\Lambda)$. Thanks to 
the fact that $s\mapsto \sqrt{s}f(s)\in L^\infty(0,\infty;L^p(\Omega,\Lambda))$, ${\mathbb{Q}}f\in L^q(0,T; L^p(\Omega,\Lambda))$
for $1<q<2$.

It remains to show that $w_2=\underline{d}\,W$ satisfies the equation $\partial w_2-\Delta w_2=df$ in $(0,\infty)\times\Omega$, 
the initial condition $w_2(0)=0$ in $\Omega$, and the boundary conditions
$\nu \wedge w_2=0$ on $(0,\infty)\times\partial\Omega$ and $\nu\wedge \delta w_2=\nu\wedge f$ on $(0,\infty)\times\partial\Omega$,
and that $w_2\in {\mathscr{C}}([0,\infty),X)$.

The first equation is satisfied since $d{\mathbb{Q}}f=df$. The initial condition is clearly satisfied since $W(0)=0$. The first
boundary condition $\nu\wedge w_2=0$ on $(0,\infty)\times\partial\Omega$ comes from \eqref{eq:nuxdu}. The only involved part is to
show that $\nu\wedge \delta w_2=\nu\wedge f$ on $(0,\infty)\times\partial\Omega$: first, since $\nu\wedge \delta W=0$ ($W(t)\in {\sf D}_p(M)$
for almost all $t>0$), 
we have that $\nu\wedge d\delta W=0$ thanks to \eqref{eq:nuxdu}. This gives then that $\nu\wedge \delta w_2=\nu\wedge MW$ 
on $(0,\infty)\times\partial\Omega$. Now, Duhamel's formula gives that $MW=-\partial_tW+{\mathbb{Q}}f$: since $\nu\wedge W=0$, we
have that $\nu\wedge\partial_tW=0$ and thanks to the definition of the projection ${\mathbb{Q}}$, we have that 
$\nu\wedge {\mathbb{Q}}f=\nu\wedge f$. This yields $\nu\wedge \delta w_2=\nu\wedge MW=\nu\wedge f$ on $(0,\infty)\times\partial\Omega$.

The fact that $w_2$ belongs to ${\mathscr{C}}([0,\infty),X)$ comes from the following estimates. Let $0<t'<t$, we have that
\begin{align*}
w_2(t)-w_2(t')&=\underline{d}\int_0^t e^{-(t-s)M}{\mathbb{Q}}f(s)\,{\rm d}s-\underline{d}\int_0^{t'} e^{-(t'-s)M}{\mathbb{Q}}f(s)\,{\rm d}s\\
&=\underline{d}\int_0^{t'} \bigl(e^{-(t-s)M}-e^{-(t'-s)M}\bigr){\mathbb{Q}}f(s)\,{\rm d}s +\underline{d}\int_{t'}^t e^{-(t-s)M}{\mathbb{Q}}f(s)\,{\rm d}s\\
&=\int_0^{t'}\underline{d}\,e^{-(t'-s)M} \bigl(e^{-(t-t')M}-{\rm Id}\bigr){\mathbb{Q}}f(s)\,{\rm d}s
+\int_{t'}^t\underline{d}\,e^{-(t-s)M}{\mathbb{Q}}f(s)\,{\rm d}s
\end{align*} 
The first term is estimated as follows:
for all $s\in (0,t']$, $\bigl(e^{-(t-t')M}-{\rm Id}\bigr){\mathbb{Q}}f(s)\xrightarrow[t'\to t]{}0$ and for all $s\in (0,t')$,
$\bigl\|\underline{d}\,e^{-(t'-s)M} \bigl(e^{-(t-t')M}-{\rm Id}\bigr){\mathbb{Q}}f(s)\bigr\|_X\le (C+1)\,\frac{C}{\sqrt{t'-s}\sqrt{s}}\,
\|s\mapsto \sqrt{s}\,f(s)\|_{L^\infty(0,\infty;L^p(\Omega,\Lambda))}$, so that 
$s\mapsto \underline{d}\,e^{-(t'-s)M} \bigl(e^{-(t-t')M}-{\rm Id}\bigr){\mathbb{Q}}f(s)\,1\hspace{-2pt}{\rm l}_{(0,t')}(s)$
is integrable on $(0,\infty)$ with $L^1$ norm independent of $t'$. Therefore, 
\[
\Bigl\|\int_0^{t'}\underline{d}\,e^{-(t'-s)M} \bigl(e^{-(t-t')M}-{\rm Id}\bigr){\mathbb{Q}}f(s)\,{\rm d}s\Bigr\|_X\xrightarrow[t'\to t]{}0.
\]
As for the second term, we proceed as follows:
\begin{align*}
\Bigl\|\int_{t'}^t\underline{d}\,e^{-(t-s)M}{\mathbb{Q}}f(s)\,{\rm d}s\Bigr\|_X\le&
\Bigl(\int_{t'}^t\frac{C}{\sqrt{t-s}\sqrt{s}}\,{\rm d}s \Bigr) 
\bigl\|s\mapsto \sqrt{s}\,{\mathbb{Q}}f(s)\bigr\|_{L^\infty(0,\infty;L^p(\Omega,\Lambda))}\\
\le &C \Bigl(\int_{\frac{t'}{t}}^1\frac{1}{\sqrt{1-s}\sqrt{s}}\,{\rm d}s \Bigr)
\bigl\|s\mapsto \sqrt{s}\,{\mathbb{Q}}f(s)\bigr\|_{L^\infty(0,\infty;L^p(\Omega,\Lambda))}\\
&\xrightarrow[t'\to t]{}0.
\end{align*} 
Note that similar computations hold if $t<t'$.
This gives $\|w_2(t)-w_2(t')\|_X\xrightarrow[t'\to t]{}0$, and then the continuity of $w_2$ with values in $X$. 

The problem is linear, so it is clear that $w:=w_1+w_2\in {\mathscr{C}}([0,\infty);X)$ is a solution of \eqref{eq:Glhp1}. 
Uniqueness follows from the fact that $u=0$ is the unique solution in ${\mathscr{C}}([0,\infty);X)$ of the problem
\[
\left\{
\begin{array}{rclcl}
\partial_t u-\Delta u&=&0&\mbox{ in }&(0,\infty)\times\Omega\\
\nu\wedge u&=&0&\mbox{ on }&(0,\infty)\times\partial\Omega\\
\nu\wedge\delta u&=&0&\mbox{ on }&(0,\infty)\times\partial\Omega\\
u(0)&=&0&\mbox{ in }&\Omega,
\end{array}
\right.
\]
(or equivalently, $\partial_tu+Mu=0$, $u(0)=0$)
since $-M$ generates an analytic semigroup on $X$. 

$(ii)$ The proofs of $(i)$ and $(ii)$ are similar, changing $\delta$ with $d$, $\underline{d}$ with $\underline{\delta}$, 
${\mathbb{Q}}$ with ${\mathbb{P}}$, $M$ with $\tilde{M}$.
\end{proof}

\section{The nonlinear Hall problem in subcritical spaces}
\label{sec:HPn}

In this section, we assume that $\Omega$ is a smooth domain in ${\mathds{R}}^n$ ($n\ge 2$). The nonlinear Hall problem
for a magnetic field $b:[0,T)\times\Omega\ni (t,x)\mapsto b(t,x)\in\Lambda^2$ ($0<T < \infty$) reads as follows:
\begin{equation}
\label{eq:HPn}
\tag{HP$_n$}
\left\{\begin{array}{rclcl}
\partial_t b-\Delta b&=&-d\,(\delta b\lrcorner\,b)&\mbox{ in }&(0,T)\times\Omega\\
\nu\wedge b&=&0&\mbox{ on }&(0,T)\times\partial\Omega\\
\nu\wedge\delta b &=&-\nu\wedge(\delta b\lrcorner\,b)&\mbox{ on }&(0,T)\times\partial\Omega\\
b(0,\cdot)&=&b_0&\mbox{ in }&\Omega,
\end{array}
\right.
\end{equation}
where, as before, $\nu(x)$ denotes the exterior unit normal vector at a point $x\in\partial\Omega$. We assume
that $\underline{d}b_0=0$ so that, according to Lemma~\ref{lem:dw=0}, $\underline{d}b(t)=0$ for all $t\in(0,T)$.
Let's point out that with $f=-\delta b\lrcorner\, b$, the system \eqref{eq:HPn} has the same structure as \eqref{eq:Glhp1},
so that we will deal with mild solutions, $i.e.$ solutions of the form \eqref{eq:w}.

Now, note that the nonlinearity of the first equation of \eqref{eq:HPn} has an order 2 derivative. The homogeneity of this
equation is given by $b_\lambda(t,x)=b(\lambda^2t,\lambda x)$, $(t,x)\in (0,\infty)\times {\mathds{R}}^n$ ($\lambda>0$): 
if $b$ satisfies the first equation of \eqref{eq:HPn}, then so does $b_\lambda$. 
This implies that a space of the form $L^r_t(L^s_x)$ is critical ($i.e.$, $\|b_\lambda\|_{L^r_t(L^s_x)}=\|b\|_{L^r_t(L^s_x)}$)
only if $r=s=\infty$. The heat semigroup does not behave well in $L^\infty_x$, so an alternative solution is to 
work in Sobolev spaces $W^{\beta,q}_t(W^{\gamma,p}_x)$ that embed into $L^\infty_{t,x}$: these are not critical spaces anymore,
but subcritical. 

We choose $p\in (n,\infty)$ and $q\in(2,\infty)$, so that $\frac{n}{p}+\frac{2}{q}<1$, and $\alpha\in (0,1)$ such that 
$\frac{2}{q}<\alpha <1-\frac{n}{p}$.
We have then $W^{\frac{\alpha}{2},q}(0,T;{\sf D}_p(M^{\frac{1-\alpha}{2}})\hookrightarrow {\mathscr{C}}_b([0,T)\times \Omega,\Lambda))$
where the embedding constant $C(T)$ depends on $T$ and $C(T)\xrightarrow[T\to 0]{}0$:
\[
\|w\|_{L^{\infty}(0,T;L^\infty(\Omega,\Lambda))}\le C(T) \|w\|_{W^{\frac{\alpha}{2},q}(0,T;{\sf D}_p(M^{\frac{1-\alpha}{2}})}, \quad
w\in W^{\frac{\alpha}{2},q}(0,T;{\sf D}_p(M^{\frac{1-\alpha}{2}}).
\]
Indeed, since $\Omega$ is smooth, we have that ${\sf D}_p(M^{\frac{1}{2}})\subset W^{1,p}(\Omega,\Lambda)$ (see \cite{S95}) and by
Sobolev embeddings, $W^{1-\alpha,p}(\Omega,\Lambda)\hookrightarrow L^\infty(\Omega,\Lambda)$. Moreover, 
$W^{\frac{\alpha}{2},q}(0,T)\hookrightarrow {\mathscr{C}}_b([0,T))$ (since $\frac{1}{q}<\frac{\alpha}{2}$) with an
embedding constant depending on $T$ that goes to $0$ as $T\to 0$..

\begin{definition}
\label{def:mildsol}
Following the linear case, we say that \eqref{eq:HPn} has a mild solution $b$ on $[0,T)$ in the space 
$X_T:=W^{\frac{\alpha}{2},q}(0,T;{\sf D}_p(M^{\frac{1-\alpha}{2}}))$ if $b$ satisfies 
\[
b(t)= e^{-tM}b_0 - \underline{d} \int_0^t e^{-(t-s)M}{\mathbb{Q}}(\delta b(s)\lrcorner\,b(s))\,{\rm d}s, \quad t\in [0,T).
\]
\end{definition}

We have the following existence result in the subcritical case.

\begin{theorem}
Let $\Omega\subset {\mathds{R}}^n$ be a smooth domain and let $1<p,q<\infty$ with $\frac{2}{q}+\frac{n}{p}<1$. 
Let $b_0\in (L^p(\Omega,\Lambda^2),{\sf D}_p(M^{\frac{1}{2}}))_{1-\frac{1}{q},q}$.
Then there exists $T_0>0$ and $b\in X_{T_0}$ solution of \eqref{eq:HPn}.
\end{theorem}

\begin{proof}
This existence theorem is proved via a classical fixed point theorem. We rewrite the problem as follows: we want to find $b\in X_T$ such that
$b=a+B(b,b)$ where $a(t) = e^{-tM}b_0$ and 
\[
B(b,\beta)(t)= - \underline{d} \int_0^t e^{-(t-s)M}{\mathbb{Q}}(\delta b(s)\lrcorner\,\beta(s))\,{\rm d}s, \quad t\in [0,T).
\]
We will work in $Y_T:=X_T\cap L^q(0,T; {\sf D}_p(M^{\frac{1}{2}}))$ instead, endowed with the norm
\[
\|w\|_{Y_T}:=\|M^{\frac{1-\alpha}{2}} w\|_{W^{\frac{\alpha}{2},q}(0,T;L^p(\Omega;\Lambda^2)}
+\|M^{\frac{1}{2}}w\|_{L^q(0,T;L^p(\Omega,\Lambda^2))}, \quad w\in Y_T.
\]

The first step is to show that $a$ belongs to $Y_T$. The choice of $b_0\in (L^p(\Omega),{\sf D}_p(M^{\frac{1}{2}}))_{1-\frac{1}{q},q}$
ensures that $a\in W^{\frac{1}{2},q}(0,\infty;L^p(\Omega,\Lambda^2))\cap L^q(0,\infty; {\sf D}_p(M^{\frac{1}{2}}))$. By interpolation, we can 
prove that $W^{\frac{1}{2},q}(0,\infty;L^p(\Omega,\Lambda^2))\cap L^q(0,\infty; {\sf D}_p(M^{\frac{1}{2}}))\hookrightarrow X_\infty$, 
so that $a\in Y_\infty\subset Y_T$ and 
\[
\|a\|_{Y_\infty}\le c \|b_0\|_{(L^p(\Omega),{\sf D}_p(M^{\frac{1}{2}}))_{1-\frac{1}{q},q}}.
\]
To apply Picard's fixed point theorem it remains to show that $B:Y_T\times Y_T \to Y_T$ is a bounded linear functional.
Let $b,\beta \in Y_T$. We have that $\delta b\in L^q(0,T;L^p(\Omega,\Lambda^1)$ and 
$\beta \in Y_T\subset X_T\subset L^\infty(0,T;L^\infty(\Omega,\Lambda^2)$ so that $\delta b\lrcorner\, \beta \in L^q(0,T;L^p(\Omega,\Lambda^1)$
with
\[
\|\delta b\lrcorner\, \beta\|_{L^q(0,T;L^p(\Omega,\Lambda^1)}\lesssim C(T) \|b\|_{Y_T}\|\beta\|_{Y_T}.
\]
Since $M$ has the maximal $L^q-L^p$ property, we have that 
\[
t\mapsto \int_0^t e^{-(t-s)M}{\mathbb{Q}}\bigl(\delta b(s)\lrcorner\,\beta(s)\bigr)\,{\rm d}s \in 
L^q(0,T;{\sf D}_p(M))\cap W^{1,q}(0,T;L^p(\Omega,\Lambda^1)
\]
with the estimate
\begin{align*}
&\Bigl\|t\mapsto M\int_0^t e^{-(t-s)M}{\mathbb{Q}}\bigl(\delta b(s)\lrcorner\,\beta(s)\bigr)\,{\rm d}s 
\Bigr\|_{L^q(0,T;L^p(\Omega,\Lambda^1)}\\
&+\Bigl\|t\mapsto \int_0^t e^{-(t-s)M}{\mathbb{Q}}\bigl(\delta b(s)\lrcorner\,\beta(s)\bigr)\,{\rm d}s\Bigr\|_ {W^{1,q}(0,T;L^p(\Omega,\Lambda^1)}\\
\lesssim\ & \|\delta b\lrcorner\, \beta\|_{L^q(0,T;L^p(\Omega,\Lambda^1)}
\lesssim \ C(T) \|b\|_{Y_T}\|\beta\|_{Y_T}.
\end{align*}
By interpolation, taking into account that $M$ acting on ${\sf N}_p(\delta)$ is equal to $\delta \underline{d}$, there exists a constant $C>0$ such
that
\[
\|B(b,\beta)\|_{Y_T}\le C C(T) \|b\|_{Y_T}\|\beta\|_{Y_T}.
\]
Now, we pick $b_0\in (L^p(\Omega,\Lambda^2),{\sf D}_p(M^{\frac{1}{2}}))_{1-\frac{1}{q},q}$. Then we choose $T_0>0$ small enough such
that 
\[
\|b_0\|_{(L^p(\Omega,\Lambda^2),{\sf D}_p(M^{\frac{1}{2}}))_{1-\frac{1}{q},q}}\le \frac{1}{4cCC(T_0)}
\]
and then we have that 
$\|a\|_{Y_{T_0}}\le \frac{1}{4\|B\|_{Y_{T_0}\times Y_{T_0}\to Y_{T_0}}}$ and then Picard's fixed point theorem gives a unique solution 
$b\in Y_{T_0}$ such that $b=a+B(b,b)$.
\end{proof}



\end{document}